\documentclass[12pt]{article}
\usepackage{indentfirst}  
\usepackage{mathrsfs}
\usepackage{enumerate}
\usepackage[a4paper,textwidth=6.3in,textheight=8.7in]{geometry}  % matches e-jc style sheet
\usepackage[ruled,linesnumbered,vlined]{algorithm2e}
\usepackage{amsmath,amsthm,tikz,float}

\usepackage[colorlinks=true,citecolor=black,linkcolor=black,urlcolor=blue]{hyperref}
\usepackage{mathrsfs}

\newtheorem{definition}{Definition}
\newtheorem{theorem}{Theorem}
\newtheorem{observation}{Observation}

\newtheorem{lemma}[theorem]{Lemma}

\begin{document}

\title{A Note on Roman \{2\}-domination problem in graphs\thanks{Supported in part by National Natural Science Foundation of China (No. 11371008) and Science and Technology Commission of Shanghai Municipality(No. 18dz2271000)}}
\author{Hangdi Chen  and  Changhong Lu\\
School of Mathematical Sciences,\\
Shanghai Key Laboratory of PMMP,\\
East China Normal University,\\
Shanghai 200241, P. R. China\\
\\
Email: 471798693@qq.com\\
Email: chlu@math.ecnu.edu.cn}
\date{}
\maketitle
\begin{abstract}
For a graph $G=(V,E)$, a Roman $\{2\}$-dominating  function (R2DF)$f:V\rightarrow \{0,1,2\}$ has the property that for every vertex $v\in V$ with $f(v)=0$, either there exists a neighbor $u\in N(v)$, with $f(u)=2$, or at least two neighbors $x,y\in N(v)$ having $f(x)=f(y)=1$. The weight of a R2DF is the sum $f(V)=\sum_{v\in V}{f(v)}$, and the minimum weight of a R2DF is the Roman $\{2\}$-domination  number $\gamma_{\{R2\}}(G)$. A R2DF is independent if the set of vertices having positive function values is
an independent set. The independent Roman $\{2\}$-domination  number $i_{\{R2\}}(G)$ is the minimum weight of an independent Roman $\{2\}$-dominating function on $G$. In this paper, we show that the decision problem associated with $\gamma_{\{R2\}}(G)$ is NP-complete even when restricted to split graphs.  We  design a linear time algorithm for computing the value of $i_{\{R2\}}(T)$ for any  tree $T$. This answers an open problem raised by Rahmouni and Chellali [Independent Roman $\{2\}$-domination in graphs, Discrete Applied Mathematics 236 (2018), 408-414].  Chellali, Haynes, Hedetniemi and McRae \cite{chellali2016roman} have showed that Roman $\{2\}$-domination number can be computed for the class of trees in linear time. As a generalization, we present a linear time algorithm for solving the Roman $\{2\}$-domination problem in block graphs.
\end{abstract}

\section{Introduction}
Let $G=(V,E)$ be a simple graph. The \emph{open neighborhood} $N(v)$ of a vertex $v$ consists of the vertices adjacent to $v$ and its \emph{closed neighborhood} is $N[v]=N(v)\cup\{v\}$. We denote $N^2[v]=\bigcup_{u\in N[v]} N[u]$. For an edge $e=uv$, it is said that $u$ (resp. $v$) is incident to $e$, denoted by $u\in e$ (resp. $v\in e$). A \emph{vertex cover} of $G$ is a subset $V'\subseteq V$ such that for each edge $uv\in E$, at least one of $u$ and $v$ belongs to $V'$. 
 A \emph{Roman dominating function (RDF)} on graph $G$ is a function $f:V\rightarrow \{0,1,2\}$ satisfying the condition that every vertex $u$ for which $f(u)=0$ is adjacent to at least one vertex $v$ for which $f(v)=2$. The \emph{weight of a Roman dominating function} $f$ is the value $f(V)=\sum_{v\in V}{f(v)}$. The minimum weight of a Roman dominating function on a graph $G$ is called the \emph{Roman domination number} $\gamma_R(G)$ of $G$. Roman domination and its variations have been studied in a number of recent papers (see, for example, \cite{Chambers2009Extremal,Cockayne2004Roman,Liu2012Upper}).
%�� Introduce Roman domination  in here.��
%The open neighborhood $N(v)$ of a vertex $v$ consists
 %,of the vertices adjacent to $v$ and its closed neighborhood is $N[v]=N(v)\cup v$

%In a recent paper,
 Chellali, Haynes, Hedetniemi and McRae \cite{chellali2016roman}  introduced a variant of Roman dominating functions. For a graph $G=(V,E)$, a \emph{Roman $\{2\}$-dominating function} $f:V\rightarrow \{0,1,2\}$ has the slightly different property that only for every vertex $v\in V$ with $f(v)=0$, $f(N(v))\ge2$, that is, either there exists a neighbor $u\in N(v)$, with $f(u)=2$, or at least two neighbors $x,y\in N(u)$ have $f(x)=f(y)=1$. The \emph{weight of a  Roman $\{2\}$-dominating  function} is the sum $f(V)=\sum_{v\in V}{f(v)}$, and the minimum weight of a  Roman $\{2\}$-dominating  function $f$ is the \emph{Roman $\{2\}$-domination number}, denoted $\gamma_{\{R2\}}(G)$. Roman $\{2\}$-domination is also called Italian domination by some scholars (\cite{henning2017italian}).
  %Henning and Klostermeyer \cite{henning2017italian} characterized the trees $T$ for which $\gamma(T)+1=\gamma_{\{R2\}}(T)$ and $\gamma_{\{R2\}}(T)=2\gamma(T)$.
  Suppose that $f:V\rightarrow \{0,1,2\}$ is a $R2DF$ on a graph $G=(V,E)$. Let $V_i=\{v|f(v)=i\}$, for $i\in\{0,1,2\}$.  If $V_1\cup V_2$ is an independent set, then $f$ is called an \emph{independent Roman $\{2\}$-dominating  function (IR$2$DF)}, which was introduced by Rahmouni and Chellali  \cite{Rahmouni2018Independent} in a recent paper. %The weight of an independent Roman $\{2\}$-dominating  function is the sum $f(V)=\sum_{v\in V}{f(v)}$, and 
The minimum weight of an  independent Roman $\{2\}$-dominating  function $f$ is the \emph{independent Roman $\{2\}$-domination number}, denoted $i_{\{R2\}}(G)$. The authors in  \cite{chellali2016roman,Rahmouni2018Independent}  have showed that the associated  decision problems for Roman $\{2\}$-domination and independent Roman $\{2\}$-domination are NP-complete for bipartite graphs. %Rahmouni and Chellali
 In \cite{Rahmouni2018Independent}, the authors raised some interesting open problems, one of which is whether there is  a linear algorithm for computing $i_{\{R2\}}(T)$ for any tree $T$.

A graph $G=(V,E)$ is a \emph{split graph} if $V$ can be partitioned into $C$ and $I$, where $C$ is a clique and $I$ is an independent set of $G$.  Split graph  is an important  subclass of chordal graphs (see \cite{chang2013algorithmic}).  Chordal graph is one of the classical classes in the perfect graph theory, see the book by Golumbic \cite{golumbic1980algorithmic}. It turns out to be very important in the domination theory. 

A \emph{cut-vertex} is any vertex whose removal increases the number of connected components. A maximal connected induced subgraph without a cut-vertex is called a \emph{block} of $G$. A graph $G$ is a \emph{block graph} if every block in $G$ is a complete graph.  There are widely research on variations of domination in block graphs (see, for example,  \cite{chang1989total,chen2010labelling,pradhan2018computing,xu2006power}).

In this paper, we first show that the decision problem associated with $\gamma_{\{R2\}}(G)$ is NP-complete for split graphs. Then, we give a linear time algorithm for computing  $i_{\{R2\}}(T)$ in any  tree $T$. %Since $i_{\{R2\}}(T)=i_{r2}(T)$, such the algorithm also gives the value of $i_{r2}(T)$.
Moreover, we present a linear time algorithm for solving the Roman $\{2\}$-domination problem in block graphs.

\section{Complexity result}

\setlength{\parindent}{2em}In this section, we consider the decision problem associated with Roman \{2\}-dominating  functions. %For this purpose, we introduce some additional notations and definitions.
 %Let $G=(V, E)$ be a graph. For an edge $e=uv$, it is said that $u$ ($v$, respectively) is incident to $e$, denoted by $u\in e$ ($v\in e$, respectively).

 ROMAN \{2\}-DOMINATING FUNCTION(R2D)

{\bfseries INSTANCE:} A graph $G=(V,E)$ and a positive integer $k\le|V|$.

{\bfseries QUESTION:} Does $G$ have a Roman \{2\}-dominating function of weight at most $k$?

We show that this problem is NP-complete by reducing the well-known NP-complete problem Vertex Cover(VC) to R2D.

 VERTEX COVER(VC)

{\bfseries INSTANCE:} A graph $G=(V,E)$ and a positive integer $k\le|V|$.

{\bfseries QUESTION:} Is there a vertex cover of size $k$ or less for $G$?
% that is, a subset $V'\subseteq V$ such that $|V'|\le k$ and, for each edge $uv\in E$, at least one of $u$ and $v$ belongs to $V'$?

\begin{theorem}\label{thm1}
    R2D is NP-complete for split graphs.
\end{theorem}
\begin{proof}
    R2D is a member of NP, since we can check in polynomial time that a function $f:V\rightarrow \{0,1,2\}$ has weight at most $k$ and is a Roman $\{2\}$-dominating function. The proof is given by reducing the VC problem in general graphs to the R2D problem in split
graphs. 

 Let $G=(V,E)$ be a graph with $V=\{v_1,v_2,\cdots,v_n\}$ and  $E=\{e_1,e_2,\cdots,e_m\}$. Let $V^1=\{v_1',v_2',\cdots,v_n'\}$. We construct the graph $G'=(V',E')$ with: \[V'=V^1\cup V\cup E,\] \[E'=\{v_iv_j | v_i\neq v_j , v_i\in V, v_j\in V \}\cup \{v_iv_i' | i=1,...,n\} \cup \{ve|v\in e,e\in E\} . \]%\cup \{ve':v\in e\}
%\[V_1=V\]
%\[V_2=\{v':vv'\in E',  v\in V\}\cup E,\]
    Notice that $G'$ is a split graph whose vertex set $V'$ is the disjoint union of the clique $V$ and the independent set $V^1\cup E$. It is clear that $G'$ can be constructed in polynomial time from $G$.

    If $G$ has a vertex cover $C$ of size at most $k$, let $f: V'\rightarrow \{0,1,2\}$ be a function  defined as follows.
      \[f(v)=\begin{cases}
        2, \text{ if }  v\in C \\
        1, \text{ if } v\in V^1, v'\in V-C  \text{ and } vv'\in E'\\
        0, \text{ otherwise}
        \end{cases}\]
 It is clear that $f$ is a Roman $\{2\}$-dominating function of $G'$ with weight at most $2k+(n-k)$.

    On the other hand, suppose that $G'$ has a Roman $\{2\}$-dominating function of weight at most $2k+(n-k)$. Among all such functions, let $g=(V_0,V_1,V_2)$ be one chosen so that
    \begin{enumerate}[(C1)]
        \item $|V^1\cap V_2|$ is minimized.
        \item Subject to Conditions (C1): $|E\cap V_0|$ is maximized.
        \item Subject to Conditions (C1) and (C2): $|V\cap V_1|$ is minimized.
        \item Subject to Conditions (C1), (C2) and (C3): the weight of $g$ is minimized.
    \end{enumerate}
%the weight of $g$ is minimized.

We make the following remarks.

    \begin{enumerate}[(i)]
        \item No vertex in $V^1$  belongs to $V_2$. Indeed, suppose to the contrary that $g(v_i')=2$ for some $i$. We reassign $0$ to $v_i'$  instead of $2$ and reassign $2$ to $v_i$. Then it  provides a R2DF on $G'$ of weight at most $2k+(n-k)$ but with less vertices of $V^1$ assigned $2$, contradicting the condition (C1) in the choice of  $g$.
        \item No vertex in $E$ belongs to $V_2$. Indeed, suppose that $g(e)=2$ for some $e\in E$ and $v_j, v_k\in e$. By reassigning $0$ to $e$ instead of $2$ and reassigning $2$ to $v_j$ instead of $g(v_j)$, we obtain a R2DF on $G'$ of weight at most $2k+(n-k)$ but with more vertices of $E$ assigned $0$,  contradicting the condition (C2) in the choice of $g$.
        \item No vertex in $E$ belongs to $V_1$. Suppose that $g(e)=1$ for some $e\in E$ and $v_j, v_k\in e$. If $g(v_j')=0$, then $g(v_j)=2$(by the definition of R2DF). By reassigning $0$ to $e$ instead of $1$, we obtain a R2DF on $G'$ of weight at most $2k+(n-k)$ but with more vertices of $E$ assigned $0$,  contradicting the condition (C2) in the choice of $g$. Hence we may assume that $g(v_j')=1$(by item(i)). Clearly we can reassign $2$ to $v_j$ instead of $0$, $0$ to $v_j'$ instead of $1$ and $0$ to $e$ instead of $1$. We also obtain a R2DF on $G'$ of weight at most $2k+(n-k)$ but with more vertices of $E$ assigned $0$, contradicting the condition (C2) in the choice of $g$.
        \item No vertex in $V$ belongs to $V_1$. Suppose to the contrary that $g(v_i)=1$ for some $i$, then $g(v_i')=1$(by item(i) and the definition of R2DF). We reassign $0$ to $v_i'$ instead of $1$ and $2$ to $v_i$ instead of $1$. It provides a R2DF on $G'$ of weight at most $2k+(n-k)$ but with less vertices of $V$ assigned $1$,  contradicting the condition (C3) in the choice of $g$.
        \item If a vertex in $V$ is assigned $2$, then its neighbor in $V^1$ is assigned $0$ by the condition (C4) in the choice of $g$.
        \item If a vertex in $V$ is assigned $0$, then its neighbor in $V^1$ is assigned $1$ by the definition of R2DF and item(i).

    \end{enumerate}

     Therefore, according to the previous items, we conclude that $V^1\cap V_2=\emptyset$, $E\subseteq V_0$, and $V\cap V_1=\emptyset$. Hence $V_2\subseteq V$. Let $C=\{v|g(v)=2\}$. Since each vertex in $E\cup (V-C)$ belongs to $V_0$ in $G'$, it is clear that $C$ is a vertex cover of $G$ by the definition of R2DF. Then $g(V^1)+g(V)+g(E)=2|C|+(n-|C|)\le 2k+(n-k)$, implying that $|C|\le k$. Consequently, $C$ is a vertex cover for $G$ of size at most  $k$.

    Since the vertex cover problem is NP-complete, the Roman $\{2\}$-domination problem is NP-complete for split graphs.
\end{proof}

\section{Independent Roman $\{2\}$-domination in trees}

In this section, a linear time dynamic programming style algorithm is given to compute the exact value of independent Roman $\{2\}$-dominating number in any tree. This algorithm is constructed using the methodology of Wimer \cite{wimer1987linear}.

%Let $T=(V,E)$ be a tree with $n$ vertices. It is well known that the vertices of $T$ have an ordering $v_1,v_2,\cdots,v_n$ such that for each $1\le i\le n-1$, $v_i$ is adjacent to exactly one $v_j$ with $j>i$ (see \cite{west2001introduction}). The ordering is called a tree ordering, where the only neighbor $v_j$ with $j>i$ is called the father of $v_i$ and $v_i$ is a child of $v_j$. For each $1\le i\le n-1$, the father of $v_i$ is denoted by $F(v_i)=v_j$. %For technical reasons, we assume that $F(v_n)=v_n$.

A \emph{rooted tree} is a pair $(T,r)$ with $T$ is a tree and $r$ is a vertex of $T$. A rooted tree $(T,r)$ is trivial if $V(T)={r}$. Given two rooted trees $(T_1,r_1)$ and $(T_2,r_2)$ with $V(T_1)\cap V(T_2)=\emptyset$, the composition of them is   $(T_1,r_1)\circ (T_2,r_2)=(T,r_1)$ with $V(T)=V(T_1)\cup V(T_2)$ and $E(T)=E(T_1)\cup E(T_2)\cup \{r_1r_2\}$. It is clear that any rooted tree can be constructed recursively from trivial rooted trees using the
defined composition.

Let $f:V(T)\rightarrow \{0,1,2\}$ be a function on $T$. Then $f$ splits two functions $f_1$ and $f_2$ according to this decomposition.  We express this as follows: $(T,f,r)=(T_1,f_1,r_1)\circ (T_2,f_2,r_2)$, where $r=r_1$, $f_1=f|_{T_1}$ and $f_2=f|_{T_2}$.  For each $1\le i\le 2$, $f|_{T_i}$ is a function that $f$ restricted to the vertices of $T_i$. On the other hand, let $f_1:V(T)\rightarrow \{0,1,2\}$ (resp. $f_2$) be a function on $T_1$ (resp. $T_2$). We can define the composition as follows: $(T_1,f_1,r_1)\circ (T_2,f_2,r_2)=(T,f,r)$, where $V(T)=V(T_1)\cup V(T_2)$, $E(T)=E(T_1)\cup E(T_2)\cup \{r_1r_2\}$, $r=r_1$ and $f=f_1\circ f_2:V(T)\rightarrow \{0,1,2\}$ with $f(v)=f_i(v)$ if $v\in V(T_i)$, $i=1,2$. Suppose that $M$ and $N$ are the sets of possible tree turples. If $(T_1,f_1,r_1)\in M$ and $(T_2,f_2,r_2)\in N$, we use $M\circ N$ to denote the set of $(T,f,r)$. In our paper, we sometimes use $T-r$ to mean $T-\{r\}$. Before presenting the algorithm, let us give the following observation.

\begin{observation}
Let $f$ be an IR2DF of $T$ and $f_1=f|_{T_1}$ (resp. $f_2=f|_{T_2}$). If $f_1(r_1)\ne 0$ (resp. $f_2(r_2)\ne 0$), then $f_1$ (resp. $f_2$) is an IR2DF of $T_1$ (resp. $T_2$).  If $f_1(r_1)=0$ (resp. $f_2(r_2)=0$), then $f_1$ (resp. $f_2$) may not be an IR2DF of $T_1$ (resp. $T_2$), but $f_1$ (resp. $f_2$) restricted to the vertices of $T_1-r_1$ (resp. $T_2-r_2$) is an IR2DF of $T_1-r_1$ (resp. $T_2-r_2$).
\end{observation}

  In order to construct an algorithm for computing independent Roman $\{2\}$-domination number, we must characterize the possible tree-subset tuples $(T,f,r)$. For this purpose,
we introduce some additional notations as follows:

\noindent%$f\backslash v$=$f|_{T-\{v\}}$;\\
$IR2DF(T)=\{ f~|~ f$ is an IR2DF of $T\}$;\\
$IR2DF_r(T)=\{ f~|~ f\notin IR2DF(T) $ and $f|_{T-r}\in IR2DF(T-r)\}$.

Then we consider the following five classes:

\noindent$A=\{(T,f,r)~|~ f\in IR2DF(T)$ and $f(r)=2\}$;\\
$B=\{(T,f,r)~|~ f\in IR2DF(T)$ and $f(r)=1\}$;\\
$C=\{(T,f,r)~|~ f\in IR2DF(T)$ and $f(r)=0\}$;\\
$D=\{(T,f,r)~|~ f\in IR2DF_r(T)$ and $f(N[r])=1\}$;\\% $f(r)=0$. If $f^*(r')=1$, then $f^*$ is an IR2DF of $T'\}$;\\
$E=\{(T,f,r)~|~ f\in IR2DF_r(T)$ and $f(N[r])=0\}$.%f$ isn't an IR2DF of $T$ with $f(r)=0$. If $f^*(r')=2$, then $f^*$ is an IR2DF of $T'\}-[d]$.

Next, we provide some Lemmas.
\begin{lemma}\label{lem2}
     $A=(A\circ C)\cup (A\circ D)\cup (A\circ E)$.
\end{lemma}
\begin{proof}
Let $(T,r)=(T_1,r_1)\circ (T_2,r_2)$ and $r=r_1$. We first show that $(A\circ C)\cup (A\circ D)\cup (A\circ E)\subseteq A$.
Suppose that  $f_1$ (resp. $f_2$) is a function on $T_1$ (resp. $T_2$). Define $f$ as the function on $T$ with $f|_{T{}_1}=f_1$
and $f|_{T{}_2}=f_2$. It is clear that the following items are true.

   \begin{enumerate}[(i)]
        \item If $(T_1,f_1,r_1)\in A$ and $(T_2,f_2,r_2)\in C$, then $(T_1,f_1,r_1)\circ(T_2,f_2,r_2)\in A$.%from the construction of $(T,f,r)$, we know that $f$ is an IR2DF of $T$ and $f(r)=f(r_1)=2$. We deduce that $(T_1,f_1,r_1)\circ(T_2,f_2,r_2)\in A$.
        \item If $(T_1,f_1,r_1)\in A$ and $(T_2,f_2,r_2)\in D$, then $(T_1,f_1,r_1)\circ(T_2,f_2,r_2)\in A$.%then it is clear that $f$ is an IR2DF of $T$ and $f(r)=f(r_1)=2$. So $(T_1,f_1,r_1)\circ(T_2,f_2,r_2)\in A$.
        \item If $(T_1,f_1,r_1)\in A$ and $(T_2,f_2,r_2)\in E$, then $(T_1,f_1,r_1)\circ(T_2,f_2,r_2)\in A$.%then $f$ is an IR2DF of $T$ and $f(r)=f(r_1)=2$. We obtain $(T_1,f_1,r_1)\circ(T_2,f_2,r_2)\in A$.
    \end{enumerate}
%Therefore, $[a]\circ[c]\cup [a]\circ [d]\cup [a]\circ [e]\subseteq [a]$.

 Now we prove  that $A\subseteq(A\circ C)\cup (A\circ D)\cup (A\circ E)$. Let $(T,f,r)\in A$ and $(T,f,r)=(T_1,f_1,r_1)\circ (T_2,f_2,r_2)$, then $f_1(r_1)=f(r)=2$. Since $f\in IR2DF(T)$ and $f_1=f|_{T_1}$, then $f_1\in IR2DF(T_1)$. 
So $(T_1,f_1,r_1)\in A$. From the independence of $V_1\cup V_2$, we  have $f_2(r_2)=f(r_2)=0$. If $f_2\in IR2DF(T_2),$ 
then we obtain  $(T_2,f_2,r_2)\in C$. If $f_2\notin IR2DF(T_2),$ then $(T_2,f_2,r_2)\in D$ or $ E$.  Hence, we conclude that  $A\subseteq(A\circ C)\cup (A\circ D)\cup (A\circ E)$. %This completes the proof.Consequently,
\end{proof}

\begin{lemma}\label{lem3}
    $B=(B\circ C)\cup (B\circ D)$.
\end{lemma}

\begin{proof}
Let $(T,r)=(T_1,r_1)\circ (T_2,r_2)$ and $r=r_1$. We first show that $(B\circ C)\cup (B\circ D)\subseteq B$.
Suppose that  $f_1$ (resp. $f_2$) is a function on $T_1$ (resp. $T_2$). Define $f$ as the function on $T$ with $f|_{T{}_1}=f_1$
and $f|_{T{}_2}=f_2$.%$f|_T{}_1=f_1$
%and $f|_T{}_2=f_2$.
   \begin{enumerate}[(i)]
        \item If $(T_1,f_1,r_1)\in B$ and $(T_2,f_2,r_2)\in C$, then $(T_1,f_1,r_1)\circ(T_2,f_2,r_2)\in B$.
        \item If $(T_1,f_1,r_1)\in B$ and $(T_2,f_2,r_2)\in D$, then $(T_1,f_1,r_1)\circ(T_2,f_2,r_2)\in B$.
    \end{enumerate}
    It is easy to check the previous items. Then we conclude that $(B\circ C)\cup (B\circ D)\subseteq B$.

Next we need to show  $B\subseteq(B\circ C)\cup (B\circ D)$ . Let $(T,f,r)\in B$ and $(T,f,r)=(T_1,f_1,r_1)\circ (T_2,f_2,r_2)$, then $f_1(r_1)=f(r)=1$. It is clear that $f_1\in IR2DF(T_1).$ So we conclude that $(T_1,f_1,r_1)\in B$. From the definition of IR2DF, we must have $f_2(r_2)=f(r_2)=0$. If $f_2\in IR2DF(T_2),$ then we obtain  $(T_2,f_2,r_2)\in C$. If $f_2\notin IR2DF(T_2),$ then $f_2(N_{T_2}[r_2])=1$ and $f_2|_{T_2-r_2}\in IR2DF(T_2-r_2)$ using the fact that $(T,f,r)\in B$. Therefore, we have $f_2\in IR2DF_{r_2}(T_2)$, implying that $(T_2,f_2,r_2)\in D$. Hence, we deduce that  $B\subseteq(B\circ C)\cup (B\circ D).$
\end{proof}

\begin{lemma}\label{lem4}
 $C=(C\circ A)\cup (C\circ B)\cup (C\circ C)\cup (D\circ A)\cup (D\circ B)\cup (E\circ A)$.
\end{lemma}

\begin{proof}
Let $(T,r)=(T_1,r_1)\circ (T_2,r_2)$ and $r=r_1$. We first show that $(C\circ A)\cup (C\circ B)\cup (C\circ C)\cup (D\circ A)\cup (D\circ B)\cup (E\circ A)\subseteq C$ .
Suppose that  $f_1$ (resp. $f_2$) is a function on $T_1$ (resp. $T_2$). Define $f$ as the function on $T$ with $f|_{T{}_1}=f_1$
and $f|_{T{}_2}=f_2$. It is easy to check the following remarks by definitions.

    \begin{enumerate}[(i)]
        \item If $(T_1,f_1,r_1)\in C$ and $(T_2,f_2,r_2)\in A$, then $(T_1,f_1,r_1)\circ(T_2,f_2,r_2)\in C$. 
         \item If $(T_1,f_1,r_1)\in C$ and $(T_2,f_2,r_2)\in B$, then $(T_1,f_1,r_1)\circ(T_2,f_2,r_2)\in C$. 
         \item If $(T_1,f_1,r_1)\in C$ and $(T_2,f_2,r_2)\in C$, then $(T_1,f_1,r_1)\circ(T_2,f_2,r_2)\in C$.  
         \item If $(T_1,f_1,r_1)\in D$ and $(T_2,f_2,r_2)\in A$, then $(T_1,f_1,r_1)\circ(T_2,f_2,r_2)\in C$. 
          \item If $(T_1,f_1,r_1)\in D$ and $(T_2,f_2,r_2)\in B$, then $(T_1,f_1,r_1)\circ(T_2,f_2,r_2)\in C$. 
           \item If $(T_1,f_1,r_1)\in E$ and $(T_2,f_2,r_2)\in A$, then $(T_1,f_1,r_1)\circ(T_2,f_2,r_2)\in C$. 
    \end{enumerate}
%Hence, we deduce that $(C\circ A)\cup (C\circ B)\cup (C\circ C)\cup (D\circ A)\cup (D\circ B)\cup (E\circ A)\subseteq C$ .

Therefore, we need to prove $C\subseteq(C\circ A)\cup (C\circ B)\cup (C\circ C)\cup (D\circ A)\cup (D\circ B)\cup (E\circ A).$ Let $(T,f,r)\in C$ and $(T,f,r)=(T_1,f_1,r_1)\circ (T_2,f_2,r_2)$, then $f\in IR2DF(T)$ and $f_1(r_1)=f(r)=0$. Consider the following cases.

{\bfseries Case 1.} $f(r_2)=2.$ Using the fact that $f\in IR2DF(T)$ and $f_2=f|_{T_2}$, then $f_2\in IR2DF(T_2)$. It means that $(T_2,f_2,r_2)\in A$. If $f_1\in IR2DF(T_1)$, then we obtain that $(T_1,f_1,r_1)\in C$.
If $f_1\notin IR2DF(T_1)$, we have $(T_1,f_1,r_1)\in D$ or $ E$.

{\bfseries Case 2.} $f(r_2)=1.$ Since $f\in IR2DF(T)$ and $f_2=f|_{T_2}$, then $f_2\in IR2DF(T_2)$. So $(T_2,f_2,r_2)\in B$. If $f_1\in IR2DF(T_1)$, then we deduce  $(T_1,f_1,r_1)\in C$. If $f_1\notin IR2DF(T_1)$, therefore,  it implies that $(T_1,f_1,r_1)\in D$.

{\bfseries Case 3.} $f(r_2)=0.$ It is clear that $f_1$ and $f_2$ are both IR2DF. Then we obtain that $(T_1,f_1,r_1)\in C$ and $(T_2,f_2,r_2)\in C$.

Therefore, we obtain $C\subseteq(C\circ A)\cup (C\circ B)\cup (C\circ C)\cup (D\circ A)\cup (D\circ B)\cup (E\circ A).$
\end{proof}

\begin{lemma}\label{lem5}
   $D=(D\circ C)\cup (E\circ B)$.
\end{lemma}

\begin{proof}
Let $(T,r)=(T_1,r_1)\circ (T_2,r_2)$ and $r=r_1$. We first show that $(D\circ C)\cup (E\circ B)\subseteq D$.
Suppose that  $f_1$ (resp. $f_2$) is a function on $T_1$ (resp. $T_2$). Define $f$ as the function on $T$ with $f|_{T{}_1}=f_1$
and $f|_{T{}_2}=f_2$. It is easy to check the following remarks by definitions.

    \begin{enumerate}[(i)]
        \item  If $(T_1,f_1,r_1)\in D$ and $(T_2,f_2,r_2)\in C$, then $(T_1,f_1,r_1)\circ(T_2,f_2,r_2)\in D$. 
         \item If $(T_1,f_1,r_1)\in E$ and $(T_2,f_2,r_2)\in B$, then $(T_1,f_1,r_1)\circ(T_2,f_2,r_2)\in D$.

    \end{enumerate}

On the other hand, we show  $D\subseteq(D\circ C)\cup (E\circ B).$ Let $(T,f,r)\in D$ and $(T,f,r)=(T_1,f_1,r_1)\circ (T_2,f_2,r_2)$. Then  $f_1(r_1)=f(r)=0$. By the definition of $D$ and $f_2=f|_{T_2}$, $f_2\in IR2DF(T_2)$. Using the fact that $f(N_T[r_1])=1$, we deduce that $f(r_2)<2$. Consider the following cases.

{\bfseries Case 1.} $f(r_2)=1.$ It is clear that $(T_2,f_2,r_2)\in B$ because $f_2$ is an IR2DF of $T_2$. Since $f_1=f|_{T_1}$ and $f_1(N_{T_1}[r_1])=0$,
 we obtain $f_1|_{T_1-r_1}\in IR2DF(T_1-r_1)$.  %$f^*_1$ is an IR2DF of $T_1'$ by defining $f^*_1(r_1')=2$. 
Hence,  we have  $f_1\in IR2DF_{r_1}(T_1)$, implying that $(T_1,f_1,r_1)\in E$.

{\bfseries Case 2.} $f(r_2)=0.$ Then $f_2$ is an IR2DF of $T_2$, implying that $(T_2,f_2,r_2)\in C$. Using the fact that $f(N_T[r_1])=1$ and $f(r_2)=0$,
we know that   $f_1(N_{T_1}[r_1])=1$. It is clear that $f_1\in IR2DF_{r_1}(T_1)$.   %$f^*_1$ is an IR2DF of $T_1'$ by defining $f^*_1(r_1')=1$. 
It implies that $(T_1,f_1,r_1)\in D$.

Consequently, we deduce that $D\subseteq(D\circ C)\cup (E\circ B).$
\end{proof}

\begin{lemma}\label{lem5}
$E=E\circ C$.
\end{lemma}

\begin{proof}
Let $(T,r)=(T_1,r_1)\circ (T_2,r_2)$ and $r=r_1$. %We first show that $[e]\circ [c]\subseteq [e]$.
Suppose that  $f_1$ (resp. $f_2$) is a function on $T_1$ (resp. $T_2$). Define $f$ as the function on $T$ with $f|_{T{}_1}=f_1$
and $f|_{T{}_2}=f_2$.
If $(T_1,f_1,r_1)\in E$ and $(T_2,f_2,r_2)\in C$, then it is clear that $(T,f,r)\in E$. Hence, $(E\circ C)\subseteq E$.

On the other hand,  let $(T,f,r)\in E$ and $(T,f,r)=(T_1,f_1,r_1)\circ (T_2,f_2,r_2)$. Then  $f_1(r_1)=f(r)=0$. By the definition of $E$, we deduce that $f(r_2)=0$. Using the fact that  $(T,f,r)\in E$, we have that $f_2\in IR2DF(T_2).$  So $(T_2,f_2,r_2)\in C$.
Notice that  $(T,f,r)\in E$, we have $f_1(N_{T_1}[r_1])=0$, implying that $(T_1,f_1,r_1)\notin D$. We can easily check that $f_1\in IR2DF_{r_1}(T_1)$.
Hence, we have $(T_1,f_1,r_1)\in E$, implying that $E\subseteq(E\circ C).$
\end{proof}

Let $T=(V,E)$ be a tree with $n$ vertices. It is well known that the vertices of $T$ have an ordering $v_1,v_2,\cdots,v_n$ such that for each $1\le i\le n-1$, $v_i$ is adjacent to exactly one vertex $v_j$ with $j>i$ (see \cite{west2001introduction}). The ordering is called a tree ordering, where the only neighbor $v_j$ with $j>i$ is called the father of $v_i$ and $v_i$ is a child of $v_j$. For each $1\le i\le n-1$, the father of $v_i$ is denoted by $F(v_i)=v_j$.

The final step is to define the initial vector. In this case, for a tree, the only basis graph is a single vertex. It is easy to obtain that the initial vector is
$(2,1,\infty,\infty,0),$ where $'\infty'$ means undefined. Now, we are ready to present the algorithm.

%\vspace{4mm}
\begin{algorithm*}[H]
    \KwIn{A tree $T=(V,E)$ with a tree ordering $v_1,v_2,\cdots,v_n$.}
    \KwOut{The independent Roman $\{2\}$-dominating  number $i_{\{R2\}}(T)$.}
    \uIf{$T=K_1$}{
        \Return $i_{\{R2\}}(T)=1$\;
    }
    \For{$i:=1$ \KwTo $n$}
    {
            initialize $l [i,1..5]$ to $[2,1,\infty,\infty,0]$ \;
    }
    \For{$j:=1$ \KwTo $n-1$}
    {
        $v_k=F(v_j)$;\\

        $l[k,1]=\min\{l[k,1]+ l[j,3], l[k,1]+ l[j,4], l[k,1]+ l[j,5]\}$\;
        $l[k,2]=\min\{l[k,2]+ l[j,3], l[k,2]+ l[j,4]\}$\;
        $l[k,3]=\min\{l[k,3]+ l[j,1], l[k,3]+ l[j,2], l[k,3]+ l[j,3], l[k,4]+ l[j,1], l[k,4]+ l[j,2], l[k,5]+ l[j,1]\}$\;
        $l[k,4]=\min\{l[k,4]+ l[j,3], l[k,5]+ l[j,2]\}$\;
        $l[k,5]=\min\{l[k,5]+ l[j,3]\}$\;

    }
    \Return $i_{\{R2\}}(T)=\min\{l[n,1], l[n,2], l[n,3]\}$\;
    \caption{INDEPENDENT-ROMAN $\{2\}$-DOM-IN-TREE}
\end{algorithm*}

From the above argument, we can obtain the following theorem.

\begin{theorem}\label{thm3}
    Algorithm INDEPENDENT-ROMAN $\{2\}$-DOM-IN-TREE can output the independent Roman $\{2\}$-domination  number of any tree $T=(V,E)$ in linear time $O(n)$, where $n=|V|$. %Such algorithm also gives the value of independent $2$-rainbow domination number $i_{r2}(T)$.
\end{theorem}

\section{Roman $\{2\}$-domination in block graphs}

Let $G$ be a connected block graph and $G$ isn't a complete graph. The \emph{block-cutpoint graph} of $G$ is a bipartite graph $T_G$ in which one partite set consists of the cut-vertices of $G$, and the other has a vertex $h_i$ for each block $H_i$ of $G$. We include $vh_i$ as an edge of $T_G$ if and only if $v\in H_i$, where $h_i$ is called \emph{block-vertex}. Since two blocks in a graph share at most one vertex, then $T_G$ is a tree. We can easily get $T_G$ in linear time (see \cite{west2001introduction}). 

Let $H$ be a block and $I$ be the set of cut-vertices of $H$. We say $H$ is a block of \emph{type 0} if $|H|=|I|$ and $H$ is a block of \emph{type 1} if $|H|=|I|+1$. If $|H|\ge |I|+2$, we say $H$ is a block of \emph{type 2}. Next, we give the definition of \emph{induced Roman \{2\}-domination function}.
\begin{definition}
 A function $f_*$ is  called an induced Roman \{2\}-domination function $(R2DF_*)$ of $T_G$, if there is a R2DF $f$ of $G$, such that

\[f_*(v)=\begin{cases}
        f(v), \text{ if } v \text{ is a cut-vertex of G}\\
        f(H)-f(I), \text{ if } v(=h) \text{ is  a block-vertex of  } T_G \\
        \end{cases}\]
\end{definition}
Therefore, we can transform Roman \{2\}-domination problem on $G$ to induced Roman \{2\}-domination problem on $T_G$. Then, we show how to verify whether a given funtion of $T_G$ is a $R2DF_*$ or not.
\begin{lemma}
There exists a R2DF $f$ of weight $\gamma_{R2}(G)$, which satisfies the following conditions.

1. If $H$ is a block of type 1, $v$ isn't a cut-vertex and $v\in H$, then $f(v)\in \{0,1\}$.

2. If $H$ is a block of type 2, $v$ isn't a cut-vertex and $v\in H$, then $f(v)=0$.
\end{lemma}
\begin{proof}
Let $f$ be a R2DF of weight $\gamma_{R2}(G)$ and $u$ be a cut-vertex of $H$, where $H$ isn't a block of type 0 and $f(u)=max_{v_0\in I} f(v_0)$. If $f(v)=2$, we can reassign $0$ to $v$ and $2$ to $u$. Hence, $f(v)\in \{0,1\}$. Futhermore, if $H$ is a block of type 2, we suppose that there exists a vertex $v\in H$ such that $f(v)=1$. If $f(u)\ge 1$, then we can reassign $2$ to $u$ and $0$ to $v$, a contradiction. Suppose that $f(u)=0$, then there exists a vertex $w\in H$, such that $w$ isn't a cut-vertex and $f(w)\ge 1$. We reassign $2$ to $u$ and $0$ to $v,w,$ a contradiction.
\end{proof}

\begin{theorem}
The function $f_*$ is a $R2DF_*$ of $T_G$ with its corresponding function $f$ satisfying $Lemma$ $8$ if and only if $f_*$ satisfies the following conditions.

1. If $H$ is a block of type 1, then $f_*(h)=0$ or $1$.

2. If $H$ is a block of type 0 or 2, then $f_*(h)=0$.

3. If $v$ is a cut-vertex with $f_*(v)=0$, then $\exists u\in N^2_{T_G}(v)$ such that $f_*(u)=2$ or $\exists u_1,u_2\in N^2_{T_G}(v)$ such that $f_*(u_1)=f_*(u_2)=1$.

4. If $H$ is a block of type 1 or 2 with $f_*(h)=0$, then $\exists u\in N_{T_G}(h)$ such that $f_*(u)=2$ or $\exists u_1,u_2\in N_{T_G}(h)$ such that $f_*(u_1)=f_*(u_2)=1$.

5. The weight of $f_*$ is $\gamma_{R2}(G)$.

%\noindent Moreover, if $v$ is a cut-vertex, then $f(v)=f_*(v)$; if $v$ isn't a cut-vertex with $v\in H$, then $f(v)=f_*(h)$.
\end{theorem}
\begin{proof}
Let $f_*$ be a $R2DF_*$ of $T_G$ with its corresponding function $f$ satisfying $Lemma$ $8$. We first show $f_*$ satisfies the above conditions. It is clear that the above item 1 and item 2 are true. Suppose that $v$ is a cut-vertex of $G$ with $f_*(v)=0$, then $f(v)=0.$ If there exists a neighbor $u\in N(v)$ with $f(u)=2$, then $u$ is a cut-vertex and $u\in N_{T_G}^2[v]$. It means that $\exists u\in N_{T_G}^2[v]$ such that $f_*(u)=2$. Otherwise, there exists at least two neighbors $x,y\in N(v)$ having $f(x)=f(y)=1$. If $x$ and $y$ are cut-vertices, then we obtain $x,y\in N_{T_G}^2[v]$ having $f_*(x)=f_*(y)=1$. If at least one of $x$ and $y$ isn't a cut-vertex, without loss of generality we can assume $x$ isn't a cut-vertex and $H$ is a block containing $x$. We deduce that $H$ is a block of type 1, implying that $f_*(h)=1$. So item 3 holds. Suppose that $H$ is a block of type 1 or 2 with $f_*(h)=0$ and $I$ is the set of cut-vertices of $H$. By $Lemma$ 8, we deduce that $f(v)=0$ for each $v\in H-I$. Since $f$ is a $R2DF$, then $\exists u\in N(v)$ such that $f(u)=2$ or $\exists u_1,u_2\in N(v)$ such that $f(u_1)=f(u_2)=1$. It is clear that $u$ is a cut-vertex. So do $u_1$, $u_2$. It means that $f_*(u)=2$ and $f_*(u_1)=f_*(u_2)=1$. We obtain item 4. It is easy to check that $f_*$ satisfies item 5.

On the other hand, let $f_*$ be a function satisfying the above conditions. Define $f$ as follows.
 \[f(v)=\begin{cases}
        f_*(v), \text{ if } v \text{ is a cut-vertex}\\
        f_*(h), \text{ if } v \text{ isn't a cut-vertex with } v\in H\\
        \end{cases}\]
Then, we show $f$ is a R2DF satisfying $Lemma$ 8. Suppose that $H$ is a block and $v$ isn't a cut-vertex with $v\in H$. If $H$ is a block of type 1, by the above item 1, we have $f(v)=f_*(h)\in \{0,1\}.$ If $H$ is a block of type 2, by the above item 2, we obtain $f(v)=f_*(h)=0$. It is clear that $f(V_G)=f_*(V_{T_G})=\gamma_{R2}(G)$. 

Suppose that $v$ is a cut-vertex with $f(v)=f_*(v)=0$. If $\exists u\in N_{T_G}^2[v]$ such that $f_*(u)=2$, by items 1-3, we deduce that $u$ is a cut-vertex and $u\in N_G(v)$. Otherwise, $\exists h_1,h_2\in N_{T_G}^2[v]$ such that $f_*(h_1)=f_*(h_2)=1$. If $h_1$ and $h_2$ are both cut-vertices, then we have $h_1,h_2\in N_G(v)$ and $f(h_1)=f(h_2)=1$. If at least one of $h_1$ and $h_2$ isn't a cut-vertex, without loss of generality, we can assume $h_1$ isn't a cut-vertex and $h_1$ represent block $H_1$ in $T_G$. We deduce that $H_1$ is a block of type 1. Hence, $\exists v_1\in H_1$ and $v_1$ isn't a cut-vertex such that $f(v_1)=f_*(h_1)=1$. Therefore, we obtain $f(N(v))\ge 2$.

Suppose that $H$ is a block containing $v$ and $v$ isn't a cut-vertex with $f(v)=f_*(h)=0$. Hence, we deduce $H$ is a block of type 1 or 2. Since item 4, we have that $\exists u\in N_{T_G}(h)$ such that $f_*(u)=2$ or $\exists u_1,u_2\in N_{T_G}(h)$ such that $f_*(u_1)=f_*(u_2)=1$. It is clear that $u$ is a cut-vertex and $u\in N_G(v)$. So do $u_1,u_2$. We also obtain $f(u)=f_*(u)=2$ and $f(u_1)=f(u_2)=1$. Therefore, we deduce $f(N(v))\ge 2$.
\end{proof}
By $Theorem$ $9$, we can easily verify whether a given function of $T_G$ is a $R2DF_*$. Then, we continue to use the method of tree composition and decomposition in Section 3. For convenience, $T_G$ is denoted by $T$ if there is no ambiguity.
 Suppose that $T$ is a tree rooted at $r$ and $f:V(T)\rightarrow \{0,1,2\}$ is a function on $T$.  $T'$ is defined as a new tree rooted at $r'$  and $f':V(T')\rightarrow \{0,1,2\}$ is  a function  on $T'$,  where $V(T')=V(T)\cup \{r'\}$ and $E(T')=E(T)\cup \{rr'\}$,  $f'|_T=f$.

  In order to construct an algorithm for computing  Roman $\{2\}$-domination number, we must characterize the possible tree-subset tuples $(T,f,r)$. For this purpose,
we introduce some additional notations as follows:

\noindent%$f\backslash v$=$f|_{T-\{v\}}$;\\
$CVX(T)=\{ r~|~ r$ is a cut-vertex of $G\}$;\\
$BVX(T)=\{ r~|~ r$ is a block-vertex of $T\}$;\\
$R2DF_*(T)=\{ f~|~ f$ is a $R2DF_*$ of $T\}$;\\
$F_1(T)=\{ f~|~ f\in R2DF_*(T)$ with $f(r)=1\}$;\\
$F_2(T)=\{ f~|~ f\in R2DF_*(T)$ with $f(r)=2\}$;\\
$R2DF_*(T^{+1})=\{ f~|~ f\notin R2DF_*(T) $, $f'\in F_1(T')$ and $f'|_T=f\}$;\\
$R2DF_*(T^{+2})=\{ f~|~ f\notin R2DF_*(T) $, $f'\in F_2(T')$ and $f'|_T=f\}-R2DF_*(T^{+1})$.%$iR2DF(T^{-v})=\{ f~|~ f\notin iR2DF(T) $ and $f\backslash v\in iR2DF(T-\{v\})\}$.

Then we consider the following eleven classes:

\noindent$A_1=\{(T,f,r)~|~ f\in R2DF_*(T),$ $r\in CVX(T)$ and $f(r)=2\}$;\\
$A_2=\{(T,f,r)~|~ f\in R2DF_*(T),$ $r\in CVX(T)$ and $f(r)=1\}$;\\
$A_3=\{(T,f,r)~|~ f\in R2DF_*(T),$ $r\in CVX(T)$ and $f(r)=0\}$;\\
$A_4=\{(T,f,r)~|~ f\in R2DF_*(T^{+1}),$ $r\in CVX(T)\}$;\\
$A_5=\{(T,f,r)~|~ f\in R2DF_*(T^{+2}),$ $r\in CVX(T)\}$;\\
$B_1=\{(T,f,r)~|~ f\in R2DF_*(T),$ $r\in BVX(T)$ and $f(N[r])\ge2\}$;\\
$B_2=\{(T,f,r)~|~ f\in R2DF_*(T),$ $r\in BVX(T)$ and $f(N[r])=1\}$;\\
$B_3=\{(T,f,r)~|~ f\in R2DF_*(T),$ $r\in BVX(T)$ and $f(N[r])=0\}$;\\
$B_4=\{(T,f,r)~|~ f\in R2DF_*(T^{+1}),$ $r\in BVX(T)$ and $f(N[r])=1\}$;\\%$B_4=\{(T,f,r)~|~ f\in R2DF_*(T^{+1}),$ $r\in BVX(T)\}$;$B_4=\{(T,f,r)~|~ f\in iR2DF(T^{-r}),$ $r\in BVX(T)$ and $f(N[r])=1\}$;\\
$B_5=\{(T,f,r)~|~ f\in R2DF_*(T^{+1}),$ $r\in BVX(T)$ and $f(N[r])=0\}$;\\
$B_6=\{(T,f,r)~|~ f\in R2DF_*(T^{+2}),$ $r\in BVX(T)\}$.%$B_5=\{(T,f,r)~|~ f\in iR2DF(T^{-r}),$ $r\in BVX(T)$ and $f(N[r])=0\}$.
%f$ isn't an IR2DF of $T$ with $f(r)=0$. If $f^*(r')=2$, then $f^*$ is an IR2DF of $T'\}-[d]$.

In order to give the algorithm, we present the following Lemmas.
\begin{lemma}\label{lem10}
     $A_1=(A_1\circ B_1)\cup (A_1\circ B_2)\cup (A_1\circ B_3)\cup (A_1\circ B_4)\cup (A_1\circ B_5)\cup (A_1\circ B_6)$.%$A_1=(A_1\circ B_1)\cup (A_1\circ B_2)\cup (A_1\circ B_3)\cup (A_1\circ B_4)\cup (A_1\circ B_5)$.
\end{lemma}
\begin{proof}
Let $(T,r)=(T_1,r_1)\circ (T_2,r_2)$ and $r=r_1$. We first show that $(A_1\circ B_1)\cup (A_1\circ B_2)\cup (A_1\circ B_3)\cup (A_1\circ B_4)\cup (A_1\circ B_5)\cup (A_1\circ B_6)\subseteq A_1$.
Suppose that  $f_1$ (resp. $f_2$) is a function on $T_1$ (resp. $T_2$). Define $f$ as the function on $T$ with $f|_{T{}_1}=f_1$
and $f|_{T{}_2}=f_2$. For each $1\le i\le 6$, if $(T_1,f_1,r_1)\in A_1$ and $(T_2,f_2,r_2)\in B_i$, it is clear that $f$ is a $R2DF_*$ of $T$, $r\in CVX(T)$ and $f(r)=f(r_1)=2$. We deduce that $(T_1,f_1,r_1)\circ(T_2,f_2,r_2)\in A_1$. It means that $(A_1\circ B_1)\cup (A_1\circ B_2)\cup (A_1\circ B_3)\cup (A_1\circ B_4)\cup (A_1\circ B_5)\cup (A_1\circ B_6)\subseteq A_1$.

 Now we prove  that $A_1\subseteq(A_1\circ B_1)\cup (A_1\circ B_2)\cup (A_1\circ B_3)\cup (A_1\circ B_4)\cup (A_1\circ B_5)\cup (A_1\circ B_6)$. Let $(T,f,r)\in A_1$ and $(T,f,r)=(T_1,f_1,r_1)\circ (T_2,f_2,r_2)$, then $f_1(r_1)=f(r)=2$. Since $f\in R2DF_*(T)$ and $f_1=f|_{T_1}$, $f_1\in R2DF_*(T_1)$ and $r_1\in CVX(T_1)$. So  $(T_1,f_1,r_1)\in A_1$ and $r_2\in BVX(T_2)$. 
If $f_2\in R2DF_*(T_2)$, then we obtain  $(T_2,f_2,r_2)\in B_1,$ $B_2$ or $B_3$. If $f_2\notin R2DF_*(T_2)$, then $(T_2,f_2,r_2)\in B_4,$ $B_5$ or $B_6$. Hence, we conclude that  $A_1\subseteq(A_1\circ B_1)\cup (A_1\circ B_2)\cup (A_1\circ B_3)\cup (A_1\circ B_4)\cup (A_1\circ B_5)\cup (A_1\circ B_6)$. 
\end{proof}

\begin{lemma}\label{lem11}
    $A_2=(A_2\circ B_1)\cup (A_2\circ B_2)\cup(A_2\circ B_3)\cup(A_2\circ B_4)\cup(A_2\circ B_5)$.%$A_2=(A_2\circ B_1)\cup (A_2\circ B_2)\cup(A_2\circ B_3)\cup(A_2\circ B_4)$.
\end{lemma}

\begin{proof}
Let $(T,r)=(T_1,r_1)\circ (T_2,r_2)$ and $r=r_1$. We first show that $(A_2\circ B_1)\cup (A_2\circ B_2)\cup(A_2\circ B_3)\cup(A_2\circ B_4)\cup(A_2\circ B_5)\subseteq A_2$.
Suppose that  $f_1$ (resp. $f_2$) is a function on $T_1$ (resp. $T_2$). Define $f$ as the function on $T$ with $f|_{T{}_1}=f_1$
and $f|_{T{}_2}=f_2$. For each $1\le i\le 5$, if $(T_1,f_1,r_1)\in A_2$ and $(T_2,f_2,r_2)\in B_i$, it is clear that $f$ is a $R2DF_*$ of $T$, $r\in CVX(T)$ and $f(r)=f(r_1)=1$. We conclude that $(T_1,f_1,r_1)\circ(T_2,f_2,r_2)\in A_2$, implying that $(A_2\circ B_1)\cup (A_2\circ B_2)\cup(A_2\circ B_3)\cup(A_2\circ B_4)\cup(A_2\circ B_5)\subseteq A_2$.

Then we need to show  $A_2\subseteq(A_2\circ B_1)\cup (A_2\circ B_2)\cup(A_2\circ B_3)\cup(A_2\circ B_4)\cup(A_2\circ B_5)$. Let $(T,f,r)\in A_2$ and $(T,f,r)=(T_1,f_1,r_1)\circ (T_2,f_2,r_2)$, then $f_1(r_1)=f(r)=1$. It is clear that $f_1$ is a $R2DF_*$ of $T_1$ and $r_1\in CVX(T_1)$. So we conclude that $(T_1,f_1,r_1)\in A_2$ and $r_2\in BVX(T_2)$. If $f_2$ is a $R2DF_*$ of $T_2$, then we obtain  $(T_2,f_2,r_2)\in B_1,B_2$ or $B_3$. If $f_2$ is not a $R2DF_*$ of $T_2$, then $f_2(N_{T_2}[r_2])\le 1$ and $f_2\in R2DF_*(T_2^{+1})$ by using the fact that $(T,f,r)\in A_2$. Therefore, we have $(T_2,f_2,r_2)\in B_4$ or $B_5$. Hence, we deduce that  $A_2\subseteq(A_2\circ B_1)\cup (A_2\circ B_2)\cup(A_2\circ B_3)\cup(A_2\circ B_4)\cup(A_2\circ B_5)$.
\end{proof}

\begin{lemma}\label{lem12}
 $A_3=(A_3\circ B_1)\cup (A_3\circ B_2)\cup (A_3\circ B_3)\cup (A_4\circ B_1)\cup (A_4\circ B_2)\cup (A_5\circ B_1)$.
\end{lemma}

\begin{proof}
Let $(T,r)=(T_1,r_1)\circ (T_2,r_2)$ and $r=r_1$. We first show that $(A_3\circ B_1)\cup (A_3\circ B_2)\cup (A_3\circ B_3)\cup (A_4\circ B_1)\cup (A_4\circ B_2)\cup (A_5\circ B_1)\subseteq A_3$.
Suppose that  $f_1$ (resp. $f_2$) is a function on $T_1$ (resp. $T_2$). Define $f$ as the function on $T$ with $f|_{T{}_1}=f_1$
and $f|_{T{}_2}=f_2$. We make some remarks.

    \begin{enumerate}[(i)]
        \item For each $1\le i\le3$, if $(T_1,f_1,r_1)\in A_3$ and $(T_2,f_2,r_2)\in B_i$, then $(T_1,f_1,r_1)\circ(T_2,f_2,r_2)\in A_3$. Indeed, if $(T_1,f_1,r_1)\in A_3$ and $(T_2,f_2,r_2)\in B_i$, then $f_1$ is a $R2DF_*$ of $T_1$ and $f_2$ is a $R2DF_*$ of $T_2$. Hence, $f$ is a $R2DF_*$ of $T$, $r\in CVX(T)$ and $f(r)=0$. Then We obtain $(T_1,f_1,r_1)\circ(T_2,f_2,r_2)\in A_3$.
         \item For each $1\le i\le2$, if $(T_1,f_1,r_1)\in A_4$ and $(T_2,f_2,r_2)\in B_i$, then $(T_1,f_1,r_1)\circ(T_2,f_2,r_2)\in A_3$. Indeed, if $(T_1,f_1,r_1)\in A_4$, then we have that $f_1\in R2DF_*(T_1^{+1})$, $r\in CVX(T)$, $f(r)=0$ and $f(N^2_{T_1}[r])=1$. By the definition of $B_i$, we obtain $f(N^2_{T}[r])\ge 2$ and $f\in iR2DF(T)$. It means that $(T_1,f_1,r_1)\circ(T_2,f_2,r_2)\in A_3$.
         \item If $(T_1,f_1,r_1)\in A_5$ and $(T_2,f_2,r_2)\in B_1$, then $(T_1,f_1,r_1)\circ(T_2,f_2,r_2)\in A_3$. Indeed, if $(T_1,f_1,r_1)\in A_5$, then we have that $f_1\in R2DF_*(T_1^{+2})$, $r\in CVX(T)$, $f(r)=0$ and $f(N^2_{T_1}[r])=0$. By the definition of $B_1$, we obtain $f(N^2_{T}[r])\ge 2$ and $f\in R2DF_*(T)$. It means that $(T_1,f_1,r_1)\circ(T_2,f_2,r_2)\in A_3$. 
    \end{enumerate}

Therefore, we need to prove $A_3\subseteq(A_3\circ B_1)\cup (A_3\circ B_2)\cup (A_3\circ B_3)\cup (A_4\circ B_1)\cup (A_4\circ B_2)\cup (A_5\circ B_1).$ Let $(T,f,r)\in A_3$ and $(T,f,r)=(T_1,f_1,r_1)\circ (T_2,f_2,r_2)$, then we have that $f_1(r_1)=f(r)=0$, $r_1\in CVX(T_1)$ and $f_2\in R2DF_*(T_2)$. So $r_2\in BVX(T_2)$. If $f_1\in R2DF_*(T_1)$,  then we obtain $(T_1,f_1,r_1)\in A_3$, implying that $(T_2,f_2,r_2)\in B_1,B_2$ or $B_3$.  Suppose that $f_1\notin R2DF_*(T_1)$. Consider the following cases.

{\bfseries Case 1.} $f_1(N^2_{T_1}[r_1])=1.$ Then we obtain $f_1\in R2DF_*(T_1^{+1})$, implying that $(T_1,f_1,r_1)\in A_4$. Since $(T,f,r)\in A_3$, we have $f_2(N_{T_2}[r_2])\ge1$. So $(T_2,f_2,r_2)\in B_1$ or $B_2$. 

{\bfseries Case 2.} $f_1(N^2_{T_1}[r_1])=0.$ So we have $f_1\in R2DF_*(T_1^{+2})$. Then $(T_1,f_1,r_1)\in A_5$. Since $(T,f,r)\in A_3$, we obtain $f_2(N_{T_2}[r_2])\ge2$. Hence, $(T_2,f_2,r_2)\in B_1$.

So $A_3\subseteq(A_3\circ B_1)\cup (A_3\circ B_2)\cup (A_3\circ B_3)\cup (A_4\circ B_1)\cup (A_4\circ B_2)\cup (A_5\circ B_1).$
\end{proof}

\begin{lemma}\label{lem13}
   $A_4=(A_4\circ B_3)\cup (A_5\circ B_2)$.
\end{lemma}

\begin{proof}
Let $(T,r)=(T_1,r_1)\circ (T_2,r_2)$ and $r=r_1$. We first show that $(A_4\circ B_3)\cup (A_5\circ B_2)\subseteq A_4$.
Suppose that  $f_1$ (resp. $f_2$) is a function on $T_1$ (resp. $T_2$). Define $f$ as the function on $T$ with $f|_{T{}_1}=f_1$
and $f|_{T{}_2}=f_2$. It is easy to check the following remarks by definitions.

    \begin{enumerate}[(i)]
        \item  If $(T_1,f_1,r_1)\in A_4$ and $(T_2,f_2,r_2)\in B_3$, then $(T_1,f_1,r_1)\circ(T_2,f_2,r_2)\in A_4$.
         \item If $(T_1,f_1,r_1)\in A_5$ and $(T_2,f_2,r_2)\in B_2$, then $(T_1,f_1,r_1)\circ(T_2,f_2,r_2)\in A_4$.

    \end{enumerate}
Therefore, we obtain that $(A_4\circ B_3)\cup (A_5\circ B_2)\subseteq A_4$.

On the other hand, we show  $A_4\subseteq(A_4\circ B_3)\cup (A_5\circ B_2)$. Let $(T,f,r)\in A_4$ and $(T,f,r)=(T_1,f_1,r_1)\circ (T_2,f_2,r_2)$. Then  we have that $f\in R2DF_*(T^{+1})$ and $r_1\in CVX(T_1)$, implying that $f(N_T^2[r_1])=1$. It means that $r_2\in BVX(T_2)$. By the definition of $A_4$ and $f_2=f|_{T_2}$, $f_2\in R2DF_*(T_2)$. Using the fact that $f(N_T^2[r_1])=1$, we deduce that $f_2(N[r_2])<2$. Consider the following cases.

{\bfseries Case 1.} $f_2(N[r_2])=1.$ It is clear that $(T_2,f_2,r_2)\in B_2$. Since $f_1(N^2_{T_1}[r_1])=f(N^2_T[r_1])-f_2(N[r_2])=0,$
 we obtain $(T_1,f_1,r_1)\in A_5$.

{\bfseries Case 2.} $f_2(N[r_2])=0.$ Then $(T_2,f_2,r_2)\in B_3$. Using the fact that $f_1(N^2_{T_1}[r_1])=f(N^2_T[r_1])-f_2(N[r_2])=1,$
we know $(T_1,f_1,r_1)\in A_4$.

Consequently, we deduce that $A_4\subseteq(A_4\circ B_3)\cup (A_5\circ B_2)$.
\end{proof}

\begin{lemma}\label{lem14}
$A_5=A_5\circ B_3$.
\end{lemma}

\begin{proof}

It is easy to check that $(A_5\circ B_3)\subseteq A_5$ by the definitions. On the other hand,  let $(T,f,r)\in A_5$ and $(T,f,r)=(T_1,f_1,r_1)\circ (T_2,f_2,r_2)$. Then  we obtain $f\in R2DF_*(T^{+2})$, $r_1\in CVX(T_1)$ and $f_1(N^2[r_1])=f(N^2[r])=0$. It implies that $(T_1,f_1,r_1)\in A_5$ and $r_2\in BVX(T_2)$. Using the fact that  $(T,f,r)\in A_5$, we deduce that $f_2(N[r_2])=0$ and $f_2\in R2DF_*(T_2)$. Therefore, $(T_2,f_2,r_2)\in B_3$. Then $A_5\subseteq(A_5\circ B_3).$
\end{proof}

\begin{lemma}\label{lem15}
 $B_1=(B_1\circ A_1)\cup (B_1\circ A_2)\cup (B_1\circ A_3)\cup (B_1\circ A_4)\cup (B_1\circ A_5)\cup (B_2\circ A_1)\cup (B_2\circ A_2)\cup (B_3\circ A_1)\cup (B_4\circ A_1)\cup (B_4\circ A_2)\cup (B_5\circ A_1)\cup (B_6\circ A_1).$% $B_1=(B_1\circ A_1)\cup (B_1\circ A_2)\cup (B_1\circ A_3)\cup (B_1\circ A_4)\cup (B_1\circ A_5)\cup (B_2\circ A_1)\cup (B_2\circ A_2)\cup (B_3\circ A_1)\cup (B_4\circ A_1)\cup (B_4\circ A_2)\cup (B_5\circ A_1).$
\end{lemma}

\begin{proof}
Let $(T,r)=(T_1,r_1)\circ (T_2,r_2)$ and $r=r_1$. We first show that $(B_1\circ A_1)\cup (B_1\circ A_2)\cup (B_1\circ A_3)\cup (B_1\circ A_4)\cup (B_1\circ A_5)\cup (B_2\circ A_1)\cup (B_2\circ A_2)\cup (B_3\circ A_1)\cup (B_4\circ A_1)\cup (B_4\circ A_2)\cup (B_5\circ A_1)\cup (B_6\circ A_1)\subseteq B_1$.
Suppose that  $f_1$ (resp. $f_2$) is a function on $T_1$ (resp. $T_2$). Define $f$ as the function on $T$ with $f|_{T{}_1}=f_1$
and $f|_{T{}_2}=f_2$. We make some remarks.

    \begin{enumerate}[(i)]
        \item For each $1\le i\le5$, if $(T_1,f_1,r_1)\in B_1$ and $(T_2,f_2,r_2)\in A_i$, then $(T_1,f_1,r_1)\circ(T_2,f_2,r_2)\in B_1$. It is easy to check it by the definitions of $B_1$ and $A_i$.
         \item For each $2\le i\le6$, if $(T_1,f_1,r_1)\in B_i$ and $(T_2,f_2,r_2)\in A_1$, then $(T_1,f_1,r_1)\circ(T_2,f_2,r_2)\in B_1$. We can easily check it by definitions too.
         \item For each $i\in \{2,4\}$, if $(T_1,f_1,r_1)\in B_i$ and $(T_2,f_2,r_2)\in A_2$, then $(T_1,f_1,r_1)\circ(T_2,f_2,r_2)\in B_1$. Indeed, it is clear that $f\in R2DF_*(T)$, $r\in BVX(T)$ and $f(N[r])=f_1(N[r_1])+f_2(r_2)=2.$ Hence, $(T_1,f_1,r_1)\circ(T_2,f_2,r_2)\in B_1.$ 
    \end{enumerate}

Therefore, we need to prove $B_1\subseteq(B_1\circ A_1)\cup (B_1\circ A_2)\cup (B_1\circ A_3)\cup (B_1\circ A_4)\cup (B_1\circ A_5)\cup (B_2\circ A_1)\cup (B_2\circ A_2)\cup (B_3\circ A_1)\cup (B_4\circ A_1)\cup (B_4\circ A_2)\cup (B_5\circ A_1)\cup (B_6\circ A_1).$ Let $(T,f,r)\in B_1$ and $(T,f,r)=(T_1,f_1,r_1)\circ (T_2,f_2,r_2)$, then we have that $f\in R2DF_*(T)$,  $r_1\in BVX(T_1)$ and $f(N[r])\ge 2$. It means that $r_2\in CVX(T_2)$. Consider the following cases.

{\bfseries Case 1.} $f(r_2)=2.$ Then we have $f_2\in R2DF_*(T_2)$, impling that $(T_2,f_2,r_2)\in A_1$. If $f_1\in R2DF_*(T_1)$, we obtain $(T_1,f_1,r_1)\in B_1,$ $B_2$ or $B_3$.  Suppose that  $f_1\notin R2DF_*(T_1)$,  then $f_1\in R2DF_*(T_1^{+1})$ or $f_1\in R2DF_*(T_1^{+2})$. Hence, $(T_1,f_1,r_1)\in B_4$, $B_5$ or $B_6$. 

{\bfseries Case 2.} $f(r_2)=1.$ It is clear that $(T_2,f_2,r_2)\in A_2$. We also have that $f_1(N[r_1])=f(N[r])-f_2(r_2)\ge 2-1\ge 1$. If $f_1\in R2DF_*(T_1)$, we obtain $(T_1,f_1,r_1)\in B_1$ or $B_2$.  Suppose that  $f_1\notin R2DF_*(T_1)$,  then $f_1\in R2DF_*(T_1^{+1})$. Therefore, $(T_1,f_1,r_1)\in B_4$.

{\bfseries Case 3.} $f(r_2)=0.$ Then we obtain $f_1(N[r_1])=f(N[r])-f_2(r_2)\ge 2$ and $f_1\in R2DF_*(T_1)$, implying that $(T_1,f_1,r_1)\in B_1$.   If $f_2\in R2DF_*(T_2)$, we deduce that $(T_1,f_1,r_1)\in A_3$.  Suppose that  $f_2\notin R2DF_*(T_2)$,  then $f_2\in R2DF_*(T_2^{+1})$  or $f_2\in R2DF_*(T_2^{+2})$. Therefore, $(T_2,f_2,r_2)\in A_4$ or $A_5$.

Hence, $B_1\subseteq(B_1\circ A_1)\cup (B_1\circ A_2)\cup (B_1\circ A_3)\cup (B_1\circ A_4)\cup (B_1\circ A_5)\cup (B_2\circ A_1)\cup (B_2\circ A_2)\cup (B_3\circ A_1)\cup (B_4\circ A_1)\cup (B_4\circ A_2)\cup (B_5\circ A_1)\cup (B_6\circ A_1).$
\end{proof}

\begin{lemma}\label{lem16}
 $B_2=(B_2\circ A_3)\cup (B_2\circ A_4)\cup (B_3\circ A_2)\cup (B_5\circ A_2).$%$B_2=(B_2\circ A_3)\cup (B_2\circ A_4)\cup (B_3\circ A_2).$
\end{lemma}

\begin{proof}
Let $(T,r)=(T_1,r_1)\circ (T_2,r_2)$ and $r=r_1$. We first show that $(B_2\circ A_3)\cup (B_2\circ A_4)\cup (B_3\circ A_2)\cup (B_5\circ A_2)\subseteq B_2$.
Suppose that  $f_1$ (resp. $f_2$) is a function on $T_1$ (resp. $T_2$). Define $f$ as the function on $T$ with $f|_{T{}_1}=f_1$
and $f|_{T{}_2}=f_2$. We make some remarks.

    \begin{enumerate}[(i)]
        \item For each $3\le i\le4$, if $(T_1,f_1,r_1)\in B_2$ and $(T_2,f_2,r_2)\in A_i$, then $(T_1,f_1,r_1)\circ(T_2,f_2,r_2)\in B_2$. It is easy to check it by the definitions.
         \item For each $i\in \{3,5\}$, if $(T_1,f_1,r_1)\in B_i$ and $(T_2,f_2,r_2)\in A_2$, then $(T_1,f_1,r_1)\circ(T_2,f_2,r_2)\in B_2$. Indeed, if $(T_1,f_1,r_1)\in B_i$ and $(T_2,f_2,r_2)\in A_2$,  we obtain that $f\in R2DF_*(T)$, $r\in BVX(T)$ and $f(N[r])=f_1(N[r_1])+f_2(r_2)=1.$ Hence, we deduce  $(T_1,f_1,r_1)\circ(T_2,f_2,r_2)\in B_2$.
    \end{enumerate}
According to the previous items, we deduce that $(B_2\circ A_3)\cup (B_2\circ A_4)\cup (B_3\circ A_2)\cup (B_5\circ A_2)\subseteq B_2$.

Therefore, we need to prove $B_2\subseteq(B_2\circ A_3)\cup (B_2\circ A_4)\cup (B_3\circ A_2)\cup (B_5\circ A_2).$ Let $(T,f,r)\in B_2$ and $(T,f,r)=(T_1,f_1,r_1)\circ (T_2,f_2,r_2)$, then we have that $f\in R2DF_*(T)$,  $r_1\in BVX(T_1)$ and $f(N[r])=1$. It implies $r_2\in CVX(T_2)$. Consider the following cases.

{\bfseries Case 1.} $f(r_2)=1.$ Then we have $f_1(N[r_1])=f(N[r])-f(r_2)=0$ and $f_2(r_2)=1$, implying that $f_2\in R2DF_*(T_2)$. So $(T_2,f_2,r_2)\in A_2$. If $f_1\in R2DF_*(T_1)$, we obtain $(T_1,f_1,r_1)\in B_3$.  Suppose that  $f_1\notin R2DF_*(T_1)$,  then $f_1(r_1)=0$ because $f\in R2DF_*(T)$. Since $f_1(N[r_1])=0$, we have that  $(T_1,f_1,r_1)\in B_5$.

{\bfseries Case 2.} $f(r_2)=0.$ It is clear that $f_1(N[r_1])=f(N[r])-f(r_2)=1$. Since $f_1=f|_{T_1}$ and $f\in R2DF_*(T)$, we have $f_1\in R2DF_*(T_1)$. Hence, $(T_1,f_1,r_1)\in B_2$. If $f_2\in R2DF_*(T_2)$, we deduce that $(T_2,f_2,r_2)\in A_3$.  Suppose that  $f_2\notin R2DF_*(T_2)$,  then $f_2(N^2[r_2])=1$. It implies $f_2\in R2DF_*(T_2^{+1})$. Therefore, $(T_2,f_2,r_2)\in A_4$.

Hence, $B_2\subseteq(B_2\circ A_3)\cup (B_2\circ A_4)\cup (B_3\circ A_2)\cup (B_5\circ A_2).$
\end{proof}

\begin{lemma}\label{lem17}
$B_3=B_3\circ A_3$.
\end{lemma}

\begin{proof}

It is easy to check that $(B_3\circ A_3)\subseteq B_3$ by the definitions. On the other hand,  let $(T,f,r)\in B_3$ and $(T,f,r)=(T_1,f_1,r_1)\circ (T_2,f_2,r_2)$. Then  we obtain $f_1(N[r_1])=f(N[r])=0$, $r_1\in BVX(T_1)$ and $f(r_2)=0$. It means that $r_2\in CVX(T_2)$. Since $f\in R2DF_*(T)$ and $f(r_2)=0$, we obtain that $f_1\in R2DF_*(T_1)$, implying that $(T_1,f_1,r_1)\in B_3$. Using the fact that  $f_1(N[r_1])=0$ and $f(r_2)=0$, we deduce that $f_2\in R2DF_*(T_2)$. Therefore, $(T_2,f_2,r_2)\in A_3$. Then $B_3\subseteq(B_3\circ A_3).$
\end{proof}

\begin{lemma}\label{lem18}
 $B_4=(B_2\circ A_5)\cup (B_4\circ A_3)\cup (B_4\circ A_4)\cup (B_4\circ A_5)\cup (B_6\circ A_2).$%$B_4=(B_2\circ A_5)\cup (B_3\circ A_4)\cup (B_4\circ A_3)\cup (B_4\circ A_4)\cup (B_5\circ A_2).$
\end{lemma}

\begin{proof}
Let $(T,r)=(T_1,r_1)\circ (T_2,r_2)$ and $r=r_1$. We first show that $(B_2\circ A_5)\cup (B_4\circ A_3)\cup (B_4\circ A_4)\cup (B_4\circ A_5)\cup (B_6\circ A_2)\subseteq B_4$.
Suppose that  $f_1$ (resp. $f_2$) is a function on $T_1$ (resp. $T_2$). Define $f$ as the function on $T$ with $f|_{T{}_1}=f_1$
and $f|_{T{}_2}=f_2$. It is easy to check the following remarks by definitions.

    \begin{enumerate}[(i)]
       \item If $(T_1,f_1,r_1)\in B_2$ and $(T_2,f_2,r_2)\in A_5$, then $(T_1,f_1,r_1)\circ(T_2,f_2,r_2)\in B_4$. %Indeed, if $(T_1,f_1,r_1)\in B_2$ and $(T_2,f_2,r_2)\in A_5$,  we obtain that $f\in iR2DF(T^{+1})$, $r\in BVX(T)$. Hence, we deduce  $(T_1,f_1,r_1)\circ(T_2,f_2,r_2)\in B_4$.
        \item For each $3\le i\le 5$, if $(T_1,f_1,r_1)\in B_4$ and $(T_2,f_2,r_2)\in A_i$, then $(T_1,f_1,r_1)\circ(T_2,f_2,r_2)\in B_4$. %It is easy to check it by definitions.
         \item If $(T_1,f_1,r_1)\in B_6$ and $(T_2,f_2,r_2)\in A_2$, then $(T_1,f_1,r_1)\circ(T_2,f_2,r_2)\in B_4$. %We can easily check it by definitions too.
    \end{enumerate}

Therefore, we need to prove $B_4\subseteq(B_2\circ A_5)\cup (B_4\circ A_3)\cup (B_4\circ A_4)\cup (B_4\circ A_5)\cup (B_6\circ A_2).$ Let $(T,f,r)\in B_4$ and $(T,f,r)=(T_1,f_1,r_1)\circ (T_2,f_2,r_2)$, then we have that $f\in R2DF_*(T^{+1})$,  $r_1\in BVX(T_1)$ and $f(N[r])=1$. It implies $r_2\in CVX(T_2)$. Consider the following cases.

{\bfseries Case 1.} $f(r_2)=1.$ Then we have $f_1(N[r_1])=f(N[r])-f(r_2)=0$ and $f_2(r_2)=1$, implying that $f_2\in R2DF_*(T_2)$. So $(T_2,f_2,r_2)\in A_2$ and $f_1\notin R2DF_*(T_1)$. Since $f_1(N[r_1])=0$ and $(T,f,r)\in B_4$, we obtain $(T_1,f_1,r_1)\in B_6$.  

{\bfseries Case 2.} $f(r_2)=0.$ It is clear that $f_1(N[r_1])=f(N[r])-f(r_2)=1$. If $f_2\in R2DF_*(T_2)$, we deduce that $(T_2,f_2,r_2)\in A_3$, implying $(T_1,f_1,r_1)\in B_4$. Suppose that $f_2\notin R2DF_*(T_2)$,  then $f_2(N^2[r_2])=0$ or 1. If $f_2(N^2[r_2])=0$, we obtain $(T_2,f_2,r_2)\in A_5$. Then, we have $(T_1,f_1,r_1)\in B_2$ or $B_4$. If $f_2(N^2[r_2])=1$, we obtain $(T_2,f_2,r_2)\in A_4$. Then, we have $(T_1,f_1,r_1)\in B_4$.%Since $f_1=f|_{T_1}$ and $f\in R2DF_*(T^{+1})$, we have $f_1\in R2DF_*(T_1)$. Hence, $(T_1,f_1,r_1)\in B_2$.

Hence, $B_4\subseteq(B_2\circ A_5)\cup (B_4\circ A_3)\cup (B_4\circ A_4)\cup (B_4\circ A_5)\cup (B_6\circ A_2).$
\end{proof}

\begin{lemma}\label{lem19}
 $B_5=(B_3\circ A_4)\cup (B_5\circ A_3)\cup (B_5\circ A_4).$
\end{lemma}
\begin{proof}
Let $(T,r)=(T_1,r_1)\circ (T_2,r_2)$ and $r=r_1$. We first show that $(B_3\circ A_4)\cup (B_5\circ A_3)\cup (B_5\circ A_4)\subseteq B_5$.
Suppose that  $f_1$ (resp. $f_2$) is a function on $T_1$ (resp. $T_2$). Define $f$ as the function on $T$ with $f|_{T{}_1}=f_1$
and $f|_{T{}_2}=f_2$. It is easy to check the following remarks by definitions.

    \begin{enumerate}[(i)]
       \item If $(T_1,f_1,r_1)\in B_3$ and $(T_2,f_2,r_2)\in A_4$, then $(T_1,f_1,r_1)\circ(T_2,f_2,r_2)\in B_5$. %Indeed, if $(T_1,f_1,r_1)\in B_2$ and $(T_2,f_2,r_2)\in A_5$,  we obtain that $f\in iR2DF(T^{+1})$, $r\in BVX(T)$. Hence, we deduce  $(T_1,f_1,r_1)\circ(T_2,f_2,r_2)\in B_4$.
        \item For each $3\le i\le 4$, if $(T_1,f_1,r_1)\in B_5$ and $(T_2,f_2,r_2)\in A_i$, then $(T_1,f_1,r_1)\circ(T_2,f_2,r_2)\in B_5$. %It is easy to check it by definitions.
    \end{enumerate}

Therefore, we need to prove $B_5\subseteq(B_3\circ A_4)\cup (B_5\circ A_3)\cup (B_5\circ A_4).$ Let $(T,f,r)\in B_5$ and $(T,f,r)=(T_1,f_1,r_1)\circ (T_2,f_2,r_2)$, then we have that $f\in R2DF_*(T^{+1})$,  $r_1\in BVX(T_1)$ and $f(N[r])=0$. It implies $r_2\in CVX(T_2)$ and $f_2(r_2)=f(r_2)=0$. Consider the following cases.

{\bfseries Case 1.} If $f_2\in R2DF_*(T_2)$, then we have $(T_2,f_2,r_2)\in A_3$ and $f_1\notin R2DF_*(T_1)$. Since $f_1(N[r_1])=0$ and $(T,f,r)\in B_5$, we obtain $(T_1,f_1,r_1)\in B_5$.  

{\bfseries Case 2.} If $f_2\notin R2DF_*(T_2)$, we deduce that $(T_2,f_2,r_2)\in A_4$. It is clear that $(T_1,f_1,r_1)\in B_3$ or $B_5$.

Hence, $B_5\subseteq(B_3\circ A_4)\cup (B_5\circ A_3)\cup (B_5\circ A_4).$
\end{proof}
\begin{lemma}\label{lem20}
 $B_6=(B_3\circ A_5)\cup (B_5\circ A_5)\cup (B_6\circ A_3)\cup (B_6\circ A_4)\cup (B_6\circ A_5).$%$B_5=(B_3\circ A_5)\cup (B_4\circ A_5)\cup (B_5\circ A_3)\cup (B_5\circ A_4)\cup (B_5\circ A_5).$
\end{lemma}

\begin{proof}
Let $(T,r)=(T_1,r_1)\circ (T_2,r_2)$ and $r=r_1$. We first show that $(B_3\circ A_5)\cup (B_5\circ A_5)\cup (B_6\circ A_3)\cup (B_6\circ A_4)\cup (B_6\circ A_5)\subseteq B_6$.
Suppose that  $f_1$ (resp. $f_2$) is a function on $T_1$ (resp. $T_2$). Define $f$ as the function on $T$ with $f|_{T{}_1}=f_1$
and $f|_{T{}_2}=f_2$. It is easy to check the following remarks by definitions.

    \begin{enumerate}[(i)]
        \item For each $i\in\{3,5\}$, if $(T_1,f_1,r_1)\in B_i$ and $(T_2,f_2,r_2)\in A_5$, then $(T_1,f_1,r_1)\circ(T_2,f_2,r_2)\in B_6$. 
         \item For each $3\le i\le5$, if $(T_1,f_1,r_1)\in B_6$ and $(T_2,f_2,r_2)\in A_i$, then $(T_1,f_1,r_1)\circ(T_2,f_2,r_2)\in B_6$. 
    \end{enumerate}

Therefore, we need to prove $B_6\subseteq(B_3\circ A_5)\cup (B_5\circ A_5)\cup (B_6\circ A_3)\cup (B_6\circ A_4)\cup (B_6\circ A_5).$ Let $(T,f,r)\in B_6$ and $(T,f,r)=(T_1,f_1,r_1)\circ (T_2,f_2,r_2)$, then we have that $f\in R2DF_*(T^{+2})$,  $r_1\in BVX(T_1)$ and $f(N[r])=0$. It implies $r_2\in CVX(T_2)$. Consider the following cases.

{\bfseries Case 1.} $f_1\in R2DF_*(T_1).$ Since $f_1(N[r_1])=f(N[r])=0$, we have $(T_1,f_1,r_1)\in B_3$. It implies $(T_2,f_2,r_2)\in A_5$.

{\bfseries Case 2.} $f_1\notin R2DF_*(T_1).$ Since $f_1(N[r_1])=f(N[r])=0$, then we obtain $(T_1,f_1,r_1)\in B_5$ or $B_6$. If $(T_1,f_1,r_1)\in B_5$, we have $f_1\in R2DF_*(T_1^{+1})$. Since $f\in R2DF_*(T^{+2})$, it means that $f_2\in R2DF_*(T_2^{+2})$. Then we deduce $(T_2,f_2,r_2)\in A_5$. If $(T_1,f_1,r_1)\in B_6$, we have $f_1\in R2DF_*(T_1^{+2})$. Since $(T,f,r)\in B_6$, we deduce that $f_2(r_2)=0$. So we obtain $(T_2,f_2,r_2)\in A_3$, $A_4$ or $A_5$.

Hence, $B_6\subseteq(B_3\circ A_5)\cup (B_5\circ A_5)\cup (B_6\circ A_3)\cup (B_6\circ A_4)\cup (B_6\circ A_5).$
\end{proof}

The final step is to define the initial vector. In this case, for a block-cutpoint graphs, the only basis graph is a single vertex. It is clear that  if $v$ is a cut-vertex, then the initial vector is $(2,1,\infty,\infty,0,\infty)$; if $v$ is a block-vertex and its corresponding block is a block of type 0, then the initial vector is $(\infty,\infty,0,\infty,\infty,\infty)$; if $v$ is a block-vertex and its corresponding block is a block of type 1, then the initial vector is $(\infty,1,\infty,\infty,\infty,0)$; if $v$ is a block-vertex and its corresponding block is a block of type 2, then the initial vector is $(\infty,\infty,\infty,\infty,\infty,0)$. Among them, $'\infty'$ means undefined. Now, we are ready to present the algorithm.
\\
\begin{algorithm*}[H]
    \KwIn{A connected block graph $G$($G\neq K_n$) and its corresponding block-cutpoint graph $T=(V,E)$ with a tree ordering $v_1,v_2,\cdots,v_n$.}
    \KwOut{The Roman $\{2\}$-dominating  number $\gamma_{\{R2\}}(G)$.}
    \For{$i:=1$ \KwTo $n$}{
            \uIf{$v_i$ is a cut-vertex}{
            initialize  $h [i,1..6]$ to $[2,1,\infty,\infty,0,\infty]$ \;
            }
            \uElseIf{$v_i$ is a block of type 0}{
             initialize  $h [i,1..6]$ to $[\infty,\infty,0,\infty,\infty,\infty]$ \;
            }
             \uElseIf{$v_i$ is a block of type 1}{
             initialize  $h [i,1..6]$ to $[\infty,1,\infty,\infty,\infty,0]$ \;
            }
            \uElse{
            initialize  $h [i,1..6]$ to $[\infty,\infty,\infty,\infty,\infty,0]$\;
            }
    }
    \For{$j:=1$ \KwTo $n-1$}
    {
        $v_k=F(v_j)$;\\
        \uIf{$v_k$ is a cut-vertex}{
        $h[k,1]=\min\{h[k,1]+ h[j,1], h[k,1]+ h[j,2], h[k,1]+ h[j,3], h[k,1]+ h[j,4], h[k,1]+ h[j,5], h[k,1]+ h[j,6]\}$\;
        $h[k,2]=\min\{h[k,2]+ h[j,1], h[k,2]+ h[j,2], h[k,2]+ h[j,3], h[k,2]+ h[j,4], h[k,2]+ h[j,5]\}$\;
        $h[k,3]=\min\{h[k,3]+ h[j,1], h[k,3]+ h[j,2], h[k,3]+ h[j,3], h[k,4]+ h[j,1], h[k,4]+ h[j,2], h[k,5]+ h[j,1]\}$\;
        $h[k,4]=\min\{h[k,4]+ h[j,3], h[k,5]+ h[j,2]\}$\;
        $h[k,5]=\min\{h[k,5]+ h[j,3]\}$\;
        }
        \uElse{
         $h[k,1]=\min\{h[k,1]+ h[j,1], h[k,1]+ h[j,2], h[k,1]+ h[j,3], h[k,1]+ h[j,4], h[k,1]+ h[j,5], h[k,2]+ h[j,1], h[k,2]+ h[j,2], h[k,3]+ h[j,1], h[k,4]+ h[j,1], h[k,4]+ h[j,2], h[k,5]+ h[j,1], h[k,6]+ h[j,1]\}$\;
        $h[k,2]=\min\{h[k,2]+ h[j,3], h[k,2]+ h[j,4], h[k,3]+ h[j,2], h[k,5]+ h[j,2]\}$\;
        $h[k,3]=\min\{h[k,3]+ h[j,3]\}$\;
        $h[k,4]=\min\{h[k,2]+ h[j,5], h[k,4]+ h[j,3], h[k,4]+ h[j,4], h[k,4]+ h[j,5],h[k,6]+ h[j,2] \}$\;
        $h[k,5]=\min\{h[k,3]+ h[j,4], h[k,5]+ h[j,3], h[k,5]+ h[j,4]\}$\;
        $h[k,6]=\min\{h[k,3]+ h[j,5], h[k,5]+ h[j,5], h[k,6]+ h[j,3], h[k,6]+ h[j,4], h[k,6]+ h[j,5]\}$\;
        
        }

    }
    \Return $\gamma_{\{R2\}}(G)=\min\{h[n,1], h[n,2], h[n,3]\}$\;
    \caption{ROMAN $\{2\}$-DOM-IN-BLOCK}
\end{algorithm*}
From the above argument, we can obtain the following theorem.

\begin{theorem}\label{thm3}
    Algorithm ROMAN $\{2\}$-DOM-IN-BLOCK can output the Roman $\{2\}$-domination  number of any block graphs $G=(V,E)$ in linear time $O(n)$, where $n=|V|$.
\end{theorem}

%\bibliographystyle{amsplain}
%\bibliographystyle{abbrv}
%\bibliography{refer}

%\bibliographystyle{amsplain}
%\bibliography{ref}
\end{document}